\documentclass[11pt]{amsart}
\usepackage{geometry}                
\usepackage[parfill]{parskip}    
\usepackage{amssymb, amsthm, amsmath, mathrsfs}
\usepackage{xypic}
\usepackage{color}
\usepackage{soul}
\usepackage{url}
\usepackage{pinlabel}
\usepackage{dbnsymb}

\usepackage[table]{xcolor}
\definecolor{mygray}{gray}{0.85}

\theoremstyle{plain}
\newtheorem{theorem}{Theorem}

\newtheorem{conjecture}[theorem]{Conjecture}
\newtheorem{corollary}[theorem]{Corollary}
\newtheorem{lemma}[theorem]{Lemma}

\newtheorem{definition}[theorem]{Definition}

\theoremstyle{definition}

\newtheorem{remark}[theorem]{Remark}

\newcommand{\kh}{Kh}
\newcommand{\hfk}{\widehat{HFK}}

\newcommand{\Z}{\mathbb{Z}}
\newcommand{\Ztwo}{\mathbb{Z}/ 2\mathbb{Z}}
\newcommand{\Q}{\mathbb{Q}}
\newcommand{\os}{Ozsv\'ath and Szab\'o}

\newcommand{\isom}{\cong}

\title{Symmetric unions without cosmetic crossing changes}
\author{Allison H. Moore}
\date{}                                         

\begin{document}

\maketitle

\begin{abstract}A symmetric union of two knots is a classical construction in knot theory  which generalizes connected sum, introduced by Kinoshita and Terasaka in the 1950s. We study this construction for the purpose of finding an infinite family of hyperbolic non-fibered three-bridge knots of constant determinant which satisfy the well-known cosmetic crossing conjecture. This conjecture asserts that the only crossing changes which preserve the isotopy type of a knot are nugatory. 
\end{abstract}

\section{Introduction}
\label{sec:intro}

In the 1950s, Kinoshita and Terasaka defined the \emph{union} of two knots as a generalization of a connected sum \cite{KinoshitaTerasaka}. An aesthetically appealing variation of this construction is a \emph{symmetric union}, in which the connected sum of a knot and its mirror image is modified by a certain tangle replacement, and the resulting diagram admits an axis of mirror symmetry. In this note we use symmetric unions to construct a new family of knots satisfying a well-known conjecture.

\begin{theorem}
\label{main}
There exists an infinite family of hyperbolic non-fibered three-bridge knots of fixed determinant which satisfy the cosmetic crossing conjecture. 
\end{theorem}
An embedded disk $D$ in $S^3$ intersecting $K$ twice with zero algebraic intersection number is called a \emph{crossing disk}. If $\partial D$ bounds an embedded disk in the complement of $K$, then the corresponding crossing $c$ is called \emph{nugatory} and a crossing change at $c$ preserves the isotopy type of $K$. \emph{Cosmetic} crossing changes are non-nugatory crossing changes which preserve the oriented isotopy type of the knot. The cosmetic crossing conjecture asserts that that no such crossings exist.

\begin{conjecture}[X. S. Lin]\label{conj:ccc}
If $K$ admits a crossing change at crossing $c$ which preserves the oriented isotopy class of the knot, then $c$ is nugatory.
\end{conjecture}

The cosmetic crossing conjecture also appears in the literature as the ``nugatory" crossing conjecture; see Problem 1.58 in Kirby's List \cite{Kirby:Problems}. To prove Theorem \ref{main} we will apply an obstruction of the author and Lidman.

\begin{theorem}\cite{LidmanMoore}
\label{ccc}
Let $K$ be a knot in $S^3$ whose branched double cover $\Sigma(K)$ is an L-space. If each summand of the first singular homology of $\Sigma(K)$ has square-free order, then $K$ admits no cosmetic crossing changes.
\end{theorem}
Recall that \emph{L-spaces} are the rational homology spheres with the simplest possible Heegaard Floer homology, meaning that $\operatorname{rank} \widehat{HF}(Y) = |H_1(Y;\mathbb{Z})|$. By work of \os~\cite{OS:BDCs}, knots that are reduced Khovanov homology thin have branched double covers that L-spaces. Thus Khovanov homology will be one of the tools we use to prove that the knots of Theorem \ref{main} satisfy the conditions of Theorem \ref{ccc}.

Prior to Theorem \ref{ccc}, the main classes of knots known to satisfy Conjecture \ref{conj:ccc} were fibered knots, two-bridge knots, and Whitehead doubles of prime, non-cabled knots \cite{Kalfagianni, Torisu, BK:KnotsWithout}, and it was shown in \cite{BFKP} that any genus one knot which might admit a cosmetic crossing change must be algebraically slice. The infinite family of knots we construct here are shown in Section \ref{sec:proof} to be non-alternating, non-fibered, hyperbolic, of genus two, and bridge number three. In a different direction, Theorem \ref{ccc} was applied in \cite{LidmanMoore} to settle the status of Conjecture \ref{conj:ccc} for all knots with up to nine crossings, families of pretzel knots of arbitrarily high genus, and certain knots arising as the branched sets of surgeries on strongly invertible L-space knots. In particular, the examples constructed in \cite{LidmanMoore} were of non-constant determinant. The present knots have fixed determinant and branched double covers with non-cyclic first homology. These properties differentiate them from all other knots known to satisfy Conjecture \ref{conj:ccc}, adding further variety to the landscape of knots for which this fundamental conjecture has been settled. 



\section{Symmetric unions}
\label{sec:symmetric}
Let $K$ denote an oriented knot in $S^3$. The mirror of $K$ is denoted $m(K)$. We will abuse notation and let $K$ refer to both the knot and its planar diagram. We will use $J$ to denote an oriented knot as well. Elementary rational tangles will be denoted by $T_{n}$ for $n\in\{\Z,\infty\}$, as indicated in Figure \ref{fig:elementary}.
\begin{figure}
	\labellist
	\small\hair 2pt
	\pinlabel $T_\infty$ at 70 -5
	\pinlabel $T_0$ at 270 -5
	\pinlabel $T_{-1}$ at 470 -5	
	\pinlabel $T_{1}$ at 670 -5
	\pinlabel $T_{2}$ at 870 -5		
	\endlabellist
	\includegraphics[width=7cm]{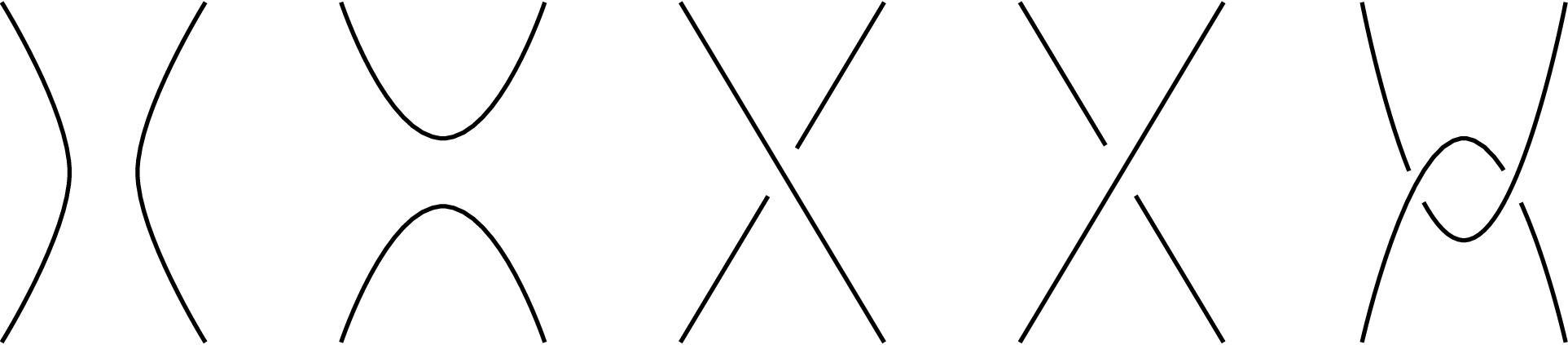}
	\caption{Examples of elementary rational tangles.}
	\label{fig:elementary}
\end{figure}
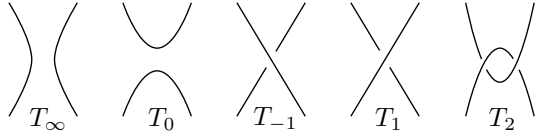
\begin{definition}
\label{def:symunion}
A \emph{symmetric union of} $J$ is an (unoriented) knot diagram obtained by replacing an elementary $0$-tangle $T_{0}$ with an elementary $n$-tangle $T_{n}$, with $n\neq 0,\infty$, along an axis of mirror symmetry in a diagram of $J\#m(J)$ as in Figure \ref{tanglerep}. A knot which admits a symmetric union diagram is called a \emph{symmetric union}, and we denote a symmetric union of $J$ by $K_n(J)$. The (unoriented) knot $J$ is called the \emph{partial knot} of $K_n(J)$, and $K_0(J)$ is $J\#m(J)$.
\end{definition}
The definition is due to Kinoshita and Teraksa \cite{KinoshitaTerasaka}. Note that when $J$ is oriented and $n$ is even, $K_n(J)$ inherits an orientation from the connected sum of $J$ with its reverse mirror image, but when $n$ is odd, the orientation of $K_n(J)$ is not well-defined. To construct an oriented symmetric union, we will adopt the convention that the north-east strand of $T_{n}\subset K_n(J)$ in Figure \ref{tanglerep} is oriented so that it agrees with the orientation of the north-east strand in $T_0 \subset K_0(J)$.\footnote{ This orientation convention is somewhat artificial; however, our choice of orientation ultimately will not matter because the knot invariants which we study in Sections \ref{sec:symmetric} and \ref{sec:proof} are not sensitive to orientation reversal.} With $K_n(J)$ oriented, the crossings in the tangle $T_{n}$ are positive whenever $n>0$.

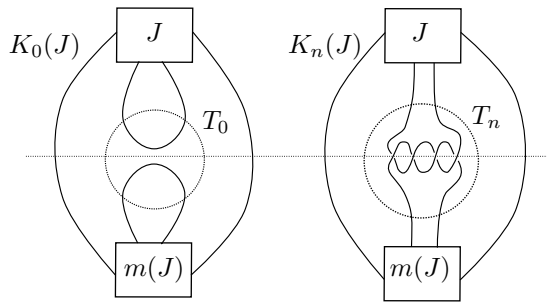
\begin{figure}
	\labellist
	\small\hair 2pt
	\pinlabel $J$ at 270 600
	\pinlabel $m(J)$ at 270 100
	\pinlabel $J$ at 830 600	
	\pinlabel $m(J)$ at 830 100
	\pinlabel $T_{0}$ at 400 410
	\pinlabel $T_{n}$ at 970 420		
	\pinlabel $K_0(J)$ at 40 570	
	\pinlabel $K_n(J)$ at 630 570	
	\endlabellist
\includegraphics[width=7cm]{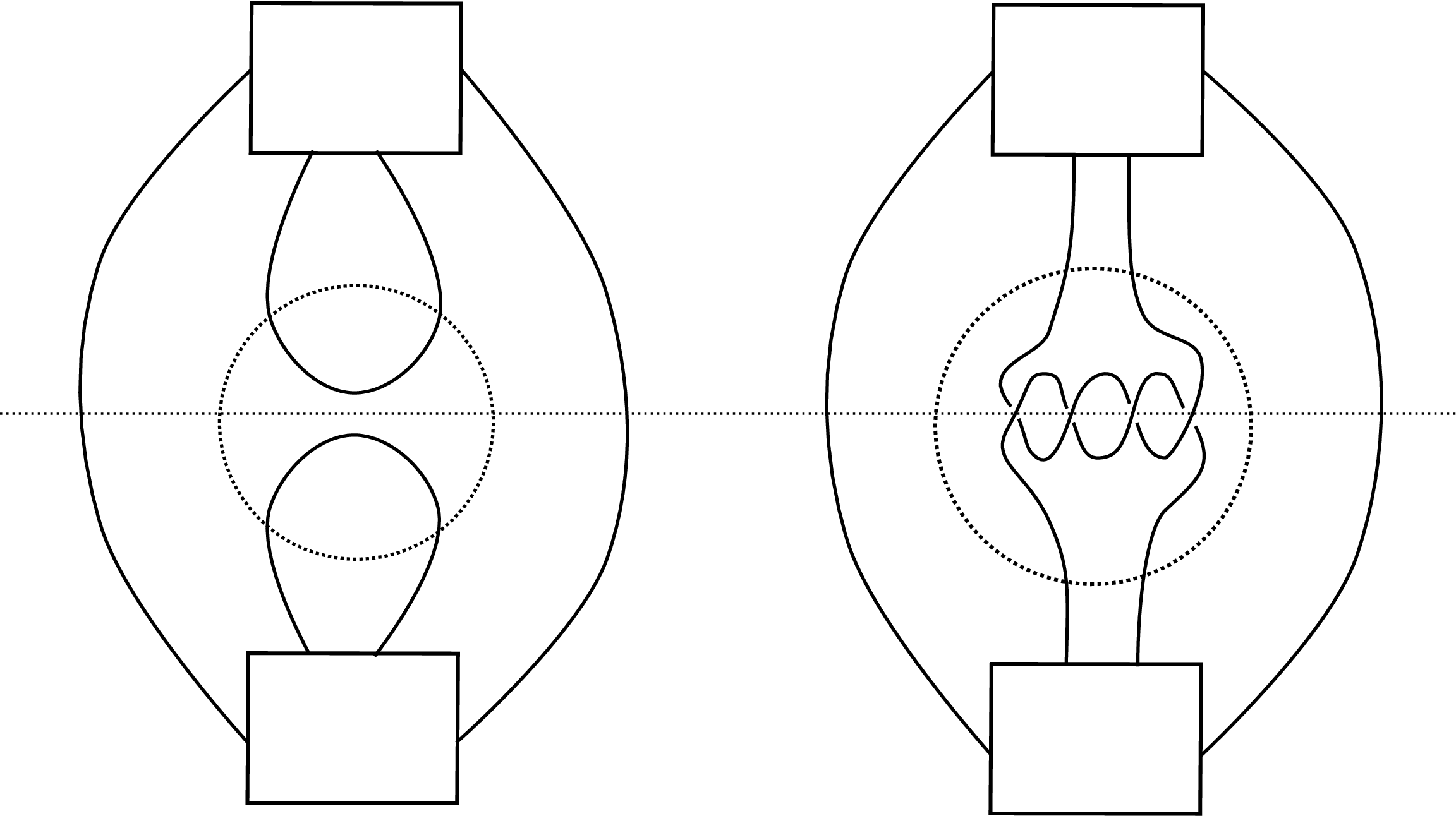}
\caption{A symmetric tangle replacement of a $0$-tangle $T_0$ with an elementary $n$-tangle $T_{n}$. (Here, $n=4$.) The diagrams of $J$ and $m(J)$ in this schematic are assumed to be mirror symmetric with respect to the horizontal axis.}
\label{tanglerep}     
\end{figure}

Elsewhere in the literature, a symmetric union may refer the generalization of this construction in which multiple symmetric tangle replacements are made, but we will call these \emph{generalized symmetric unions}. The reader is warned that the symmetric union construction is not unique; the isotopy type of $K_n(J)$ depends on both the diagram of $J\#m(J)$ and the location of the tangle replacement. For example, two distinct symmetric unions of the unknot are pictured in Figure \ref{fig:symexamples}. Despite this dependence on the diagram, a classical fact about symmetric unions is that when $n$ is even, the Alexander polynomial of $K_n(J)$ depends on neither $n$ nor the choice of diagram. 
\begin{theorem}\cite{KinoshitaTerasaka}
\label{Alexander}
If $K_n(J)$ is any symmetric union of the knot $J$ and $n$ is even, then 
\begin{equation*}
	\Delta_{K_n(J)}(t) = (\Delta_J(t))^2.
\end{equation*}
\end{theorem}
Moreover, $\det(K_n(J))=\det(J)^2$ for any $n$ (cf \cite[Theorem 2.6]{Lamm:SymmetricRibbon}). 

A symmetric union is always ribbon, which is evidenced by the existence of a symmetric ribbon disk in its symmetric diagram, similar to the one that occurs in any symmetric diagram of $J\#m(J)$. The \os~$\tau$-invariant \cite[Corollary 1.3]{OS:FourBall} gives a lower bound on the smooth four-ball genus, $|\tau(K)|\leq g_4(K)$, as does Rasmussen's $s$-invariant \cite[Theorem 1]{Rasmussen:s}. Hence these invariants will vanish for any symmetric union, a feature we will utilize in Section \ref{sec:KH}.

\begin{definition}
If replacing the elementary $n$-tangle $T_{n}$ of a symmetric union diagram $K_n(J)$ with the $\infty$-tangle $T_\infty$ results in a two-component unlink, we say that the diagram $K_n(J)$ has \emph{symmetric fusion number one}. A knot which admits a symmetric union diagram of symmetric fusion number one is also said to have symmetric fusion number one.
\end{definition}

For example, the pretzel knots of the form $(p,q,-p)$, for $p,q\in\Z$ with $p$ odd, have symmetric fusion number one. See Figure \ref{fig:bandnumberone}. Note that if $K_n(J)$ has symmetric fusion number one, then $K_m(J)$ has symmetric fusion number one for any $m\neq 0,\infty$. A knot of symmetric fusion number one is necessarily the band sum of a two-component unlink.

Generalized symmetric unions of the figure eight knot were used by Kanenobu to construct infinite families of knots with different Alexander modules and the same Jones polynomials \cite{Kanenobu:Infinitely}. More recently, Kanenobu's knots have become popular in the study of knot polynomials and knot homology theories (for instance \cite{Watson, Lobb, GW:Turaev}). The proof of Theorem \ref{main} will make use of some of the same techniques as Kanenobu \cite{Kanenobu:Infinitely} and Greene and Watson \cite{GW:Turaev}. 

\begin{figure}
\includegraphics[width=5.5cm]{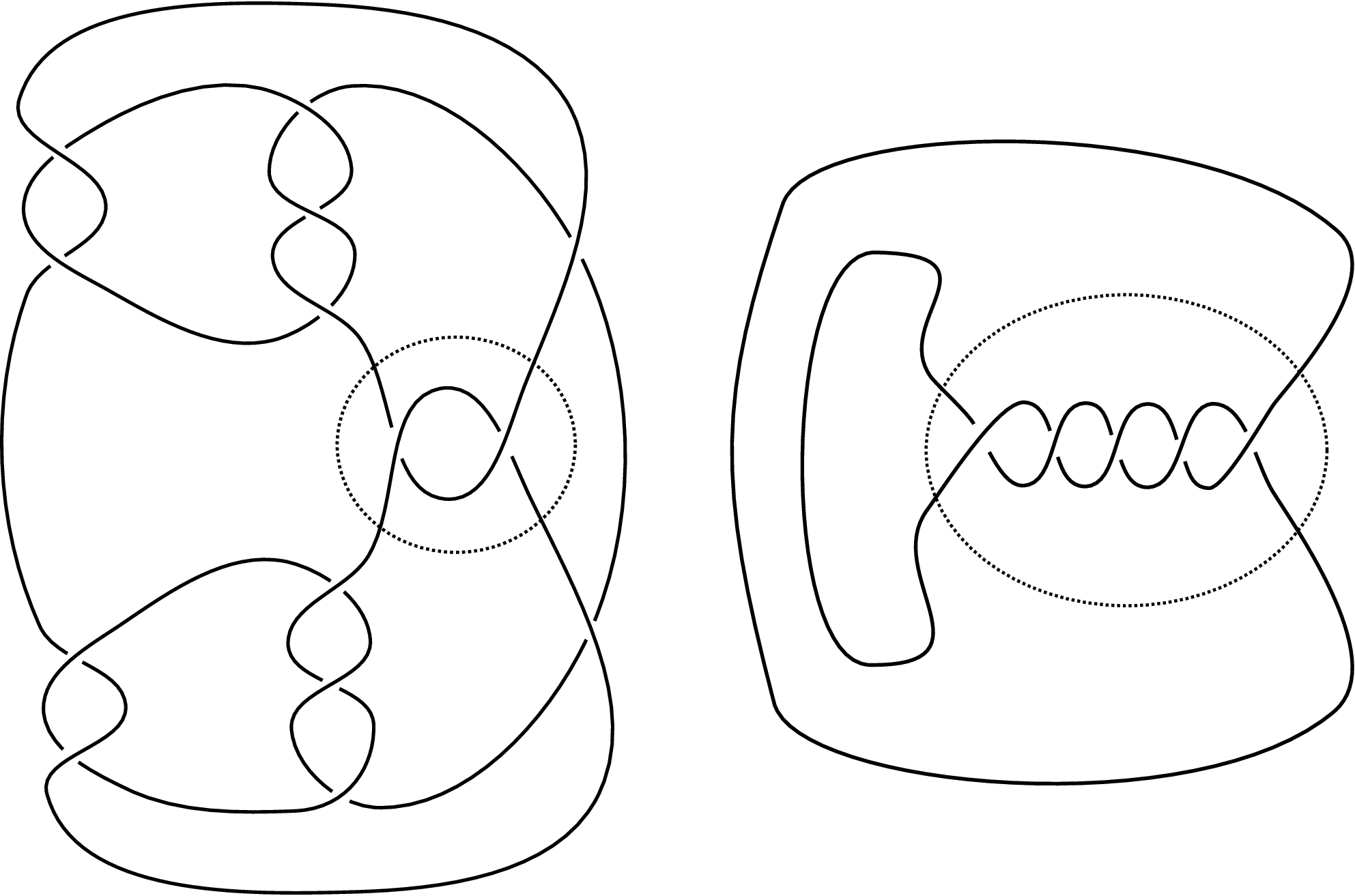}
\caption{The knot on the left is the Kinoshita-Terasaka knot $11 n 42$, and the knot on the right is an unknot. Both have partial knot the unknot.}
\label{fig:symexamples}
\end{figure}
\begin{figure}[t]
\includegraphics[width=3cm]{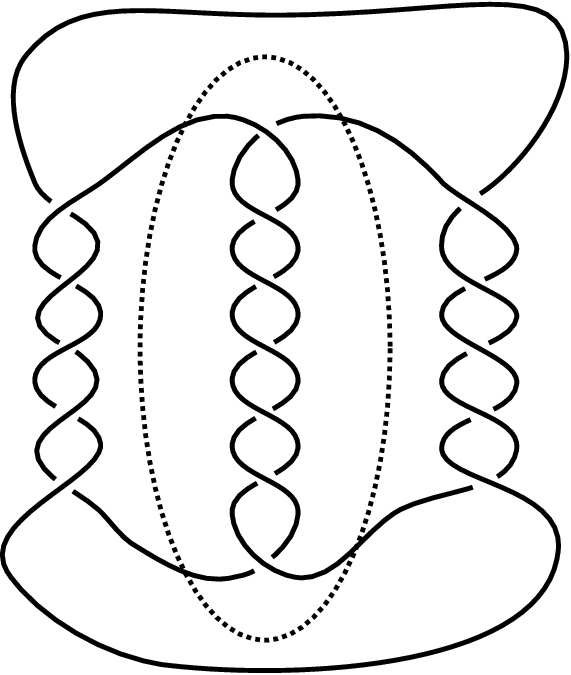}
\caption{Pretzel knots of the form $(p,q,-p)$, for $p,q\in\Z$ with $p$ odd, have symmetric fusion number one. The axis of mirror symmetry is vertical in this example.}
\label{fig:bandnumberone}
\end{figure}

\subsection{Knot Floer homology}
\label{sec:hfk}

Let $\hfk_m(K,s)$ refer to the knot Floer homology of $K\subset S^3$ with $\Z/2\Z$ coefficients, due to \os~\cite{OS:KnotInvariants} and Rasmussen \cite{Rasmussen:Thesis}. This knot invariant is a bigraded vector space with Maslov grading $m$ and Alexander grading $s$. Because knot Floer homology categorifies the symmetrized Alexander polynomial, one may wonder if a statement generalizing Theorem \ref{Alexander} holds for the knot Floer groups, and in particular whether a K\"unneth formula like the one satisfied by connected sums,
\begin{equation*}
	\hfk(K_1\# K_2) \isom \hfk(K_1)\otimes\hfk(K_2),
\end{equation*}
holds. Unfortunately no such property can hold for symmetric unions in general. Knot Floer homology detects the unknot \cite{OS:KnotInvariants}, therefore any nontrivial symmetric union of an unknot (e.g. the one in Figure \ref{fig:symexamples}) will have $\hfk(K)$ nontrivial, contradicting any general analogy. However, when $K_n(J)$ has symmetric fusion number one, Kinoshita and Terasaka's characterization of the Alexander polynomial of a symmetric union does indeed generalize.
\begin{theorem}
\label{hfk}
Let $K_n(J)$ be a symmetric union of a knot $J$ such that $K_n(J)$ has symmetric fusion number one. When $n$ is even, there is a graded isomorphism
\begin{equation*}
	\hfk(K_n(J)) \isom \hfk(J)\otimes \hfk(m(J)),
\end{equation*}
and when $n$ is odd, we have that
\begin{equation*}
	\hfk(K_n(J)) \isom \hfk(K_1(J)).
\end{equation*}
\end{theorem}
This follows as a special case of \cite[Theorem 1]{HeddenWatson} (alternatively \cite[Theorem 3.3]{MooreStarkston}) whose proof we will not repeat here. The key observation is that after perhaps mirroring, the knots $K_n(J)$ and $K_{n-2}(J)$ form an oriented skein triple with the two-component unlink and their knot Floer groups fit into a long exact sequence \cite[Theorem 1.1]{OS:Squence}. Using that symmetric unions are ribbon, hence slice, the concordance invariant $\tau(K_n(J))$ vanishes for all $n$. This fact, taken together with the skein triple and the observation that $K_0(J)$ is $J\#m(J)$, gives the statement of the theorem. 

Because knot Floer homology detects genus \cite{OS:KnotInvariants} and fiberedness \cite{Ni}, and satisfies a K\"unneth formula under connected sum, the following corollaries are immediate.
\begin{corollary}
\label{fibergenus}
Let $K_n(J)$ be a symmetric union of a knot $J$ such that $K_n(J)$ has symmetric fusion number one. If $n$ is even, then $K_n(J)$ is fibered if and only if $J$ is fibered and $g(K_n(J))=2g(J)$. If $n$ is odd, then  $g(K_n(J))=g(K_1(J))$.
\end{corollary}

\begin{corollary}
\label{notunknot}
Let $K_n(J)$ be a symmetric union of symmetric fusion number one with $n$ is even. Then $K_n(J)$ is nontrivial if and only if the partial knot $J$ is nontrivial.
\end{corollary}

\subsection{Main examples}
\label{sec:mainexamples}
We now define the main examples of interest in this note. Denote by $\mathcal{K}$ the subset of symmetric unions 
\begin{equation}
\label{family}
\{K_n \mid n\equiv0\pmod{14}, \; n\neq0\} \subset \{K_n:=K_n(5_2) \mid n\in \Z\}
\end{equation}
where the symmetric unions $K_n:=K_n(5_2)$ are constructed from the knot $5_2$ as shown in Figure \ref{sym52}. The knot $K_n(5_2)$ has symmetric fusion number one for all $n\neq 0, \infty$. 
\begin{figure}
\includegraphics[width=4cm]{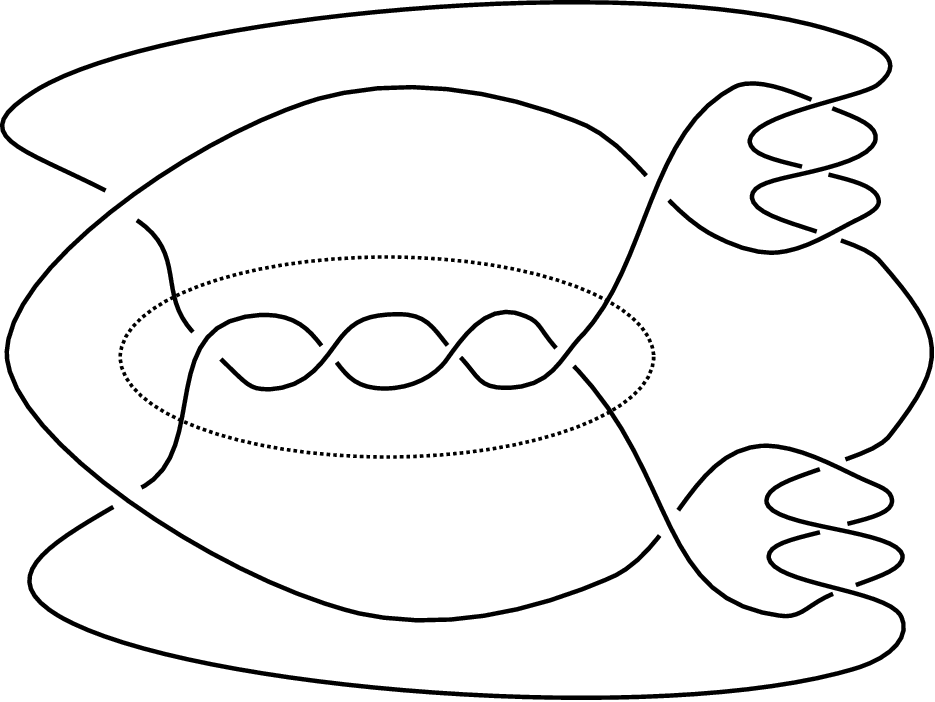}
\caption{The symmetric unions $K_n(5_2)$ of the knot $5_2$. For the knot pictured here, $n=4$.}
\label{sym52}
\end{figure}
For the remainder of this note we will assume that $K_n$ denotes the specific symmetric union $K_n(5_2)$ for $n\in \Z$.


\section{Proof of main theorem}
\label{sec:proof}
Before addressing the main theorem, we need to prove a lemma about the Khovanov homology of $K_n$. This will allow us to deduce that the branched double cover of $K_n$ is an L-space for all $n$.

\subsection{Khovanov homology}
\label{sec:KH}

Let $\kh^{q,u}(L)$ refer to the Khovanov homology of a link $L\in S^3$ with quantum grading $q$ and (co)homological grading $u$, and coefficients in $\Q$. The Khovanov homology groups are a link invariant which categorify a normalized Jones polynomial \cite{Khovanov}. Our grading and notational conventions follow Rasmussen \cite{Rasmussen:Knot}. For example, the Khovanov homology of the knot $5_2$ with all positive crossings is described by the Poincar\'e polynomial
\begin{equation}
\label{52thin}
	P_{\kh(5_2)}(q,u) = q + q^3 + q^3u +q^5u^2 +q^7u^2 + q^9u^3 + q^9u^4 +q^{13}u^5.
\end{equation}

The Khovanov thin knots are those with homology supported in a two diagonals $\delta=q-2u$ of the gradings.\footnote{When working with $\Z$--coefficients, thin knots must also have homology that is free over $\Z$.} Khovanov homology satisfies an unoriented skein exact sequence (cf \cite[Lemma 4.2]{Rasmussen:Knot}). With our conventions this is
\begin{equation}
\label{khskein}
\stackrel{\cdot u}{\longrightarrow} q^{2+3\varepsilon} u^{1+\varepsilon} \kh(\hsmoothing) \longrightarrow \kh(\overcrossing)
\longrightarrow q \kh (\smoothing) \stackrel{\cdot u}{\longrightarrow}
\end{equation}
where $\varepsilon$ is the difference between the number of negative crossings in the unoriented
resolution $(\hsmoothing)$ and the original diagram $(\overcrossing)$. As in \cite{Rasmussen:Knot}, the notation $q \kh (\smoothing)$ means the complex $\kh (\smoothing)$ is shifted in such a way as to multiply its Poincar\'e polynomial by $q$. The arrow marked with $\cdot u $ is the boundary map and it raises the homological grading by $1$. Though computations similar to Lemma \ref{khthin} can be found in \cite{MooreStarkston, Starkston}, for concreteness we provide a proof. 
\begin{lemma}
\label{khthin}
The knot $K_n$ is Khovanov homology thin with $\Q$--coefficients for all $n$. Moreover, $\kh(K_n)$ for $n\geq0$ is given by the closed formula
\begin{small}
\begin{equation}
	\begin{array}{ccccc}
\kh(K_n)&= & \mathbf{1}^{0}_{-1} + & \mathbf{1}^{n-5}_{2(n-5)-1} \cdot (\mathbf{1}^0_0 + \mathbf{1}^1_2 + \mathbf{3}^2_4 +\mathbf{3}^3_6 + \mathbf{4}^4_8 + \mathbf{4}^5_{10} + \mathbf{3}^6_{12} + \mathbf{3}^7_{14} + \mathbf{1}^8_{16} + \mathbf{1}^9_{18} ) \\
 &+& \mathbf{1}^0_1 + & \mathbf{1}^{n-4}_{2(n-4)+1} \cdot (\mathbf{1}^0_0 + \mathbf{1}^1_2 + \mathbf{3}^2_4 +\mathbf{3}^3_6 + \mathbf{4}^4_8 + \mathbf{4}^5_{10} + \mathbf{3}^6_{12} + \mathbf{3}^7_{14} + \mathbf{1}^8_{16} + \mathbf{1}^9_{18})
	\end{array}
\label{formula}
\end{equation}
\end{small}
where for brevity, $\mathbf{d}^u_q$ denotes $\Q^\mathbf{d}$ in bigrading $(q,u)$.
\end{lemma}
\begin{proof}
Without loss of generality, we assume $n\geq0$; a similar proof holds in the case that $n\leq0$ with minor changes in bigradings. Alternatively, the result for $n\leq0$ follows from the isotopy $K_{-n}(J)\simeq m(K_n(J))$ obtained by rotating about the axis of symmetry and the identity $\kh^{q,u}(m(K))\isom \kh^{-q,-u}(K)$ for all $q,u$ and any knot $K$.

We proceed by induction on $n$. 
The cases $\kh(K_n)$ for $0\leq n\leq 7$ have been verified computationally using the \emph{KnotTheory`} package for Mathematica \cite{katlas}. Assume that $n>7$. For the inductive hypothesis, $\kh(K_{n-1})$ is thin and described by \eqref{formula}.

Any crossing in the tangle $T_n$ gives rise to an unoriented triple $\{(\overcrossing), (\hsmoothing), (\smoothing)\}$, where $K_n = (\overcrossing)$ and $K_{n-1}=(\hsmoothing)$. Since $K_n$ has symmetric fusion number one for all $n\neq 0,\infty$, the resolution $(\smoothing)$ corresponds with the two-component unlink.  Because $n\geq0$, the crossings in $T_{n}$ are positive, and so the number of negative crossings in the diagram $K_n$ for any $n$ is equal to the total number of crossings in the diagram of the partial knot. Therefore the difference $\varepsilon$ in negative crossings between the resolutions $(\hsmoothing)$ and $(\overcrossing)$ is zero and the skein triple becomes
\begin{equation*}
\stackrel{\cdot u}{\longrightarrow} q^2 u \kh(\hsmoothing) \longrightarrow \kh(\overcrossing)
\longrightarrow q \kh (\smoothing) \stackrel{\cdot u}{\longrightarrow}.
\end{equation*}

The two-component unlink has Khovanov homology $\Q_{(-2)}\oplus \Q^2_{(0)} \oplus \Q_{(2)}$ supported in homological grading zero. 
Using this and the inductive hypothesis, whenever $u\neq0,1$ or $q\neq 1,3$ the sequence splits as
\begin{equation}
\label{iso}
0{\longrightarrow} q^2 u \kh(\hsmoothing) \longrightarrow \kh(\overcrossing)\longrightarrow 0,
\end{equation}
implying the isomorphism $\kh^{q,u}(\overcrossing)\isom  \kh^{q-2,u-1}(\hsmoothing)$ for all $u\neq0,1$ or $q\neq 1,3$. For $(q,u) = (1,0)$ and $(1,1)$, the sequence splits as
\begin{equation}
\label{one}
0\longrightarrow  \kh(\overcrossing)
\longrightarrow q\Q^2 \stackrel{\cdot u}{\longrightarrow} q^2 u \Q \longrightarrow \kh(\overcrossing) \rightarrow 0,
\end{equation}
and for $(q,u) = (3,0)$ and $(3,1)$, the sequence splits as
\begin{equation}
\label{two}
0\longrightarrow \kh(\overcrossing) \longrightarrow q\Q \stackrel{\cdot u}{\longrightarrow} q^2 u \Q \longrightarrow \kh(\overcrossing) \rightarrow 0.
\end{equation}
Exactness yields two solutions for each of \eqref{one} and \eqref{two}, 
\begin{equation}
	\begin{array}{ccc}
  \kh^{1,0}(\overcrossing)\oplus \kh^{1,1} &=& \Q\oplus 0 \; \text{ or } \; \Q^2 \oplus \Q \\
  \kh^{3,0}(\overcrossing)\oplus \kh^{3,1} &=& 0 \oplus 0 \; \text{ or } \; \Q\oplus \Q.
\end{array}
\label{solns}
\end{equation}
We aim to show that the first choice in each line of \eqref{solns} is the correct one, so let us assume for the contrary that the second outcome of \eqref{one} holds. 

Because symmetric unions are ribbon, and therefore slice, the concordance invariant $s(K_n(J))$ vanishes for all $n$. In particular, the Lee spectral sequence \cite{Lee} must converges to two copies of $\Q$ in quantum gradings that average to zero, hence these surviving elements live in $q=\pm1$. Suppose the two survivors are in gradings $(-1,0)$ and $(1,1)$. See Table \ref{E1}. With our current conventions, the induced differential on the $r$--th page of the Lee spectral sequence increases the homological grading by $1$ and the quantum grading by $2r$. By assumption there is a $\Q^2$ summand of the $E^1$ page in bigrading $(1,0)$, and it must cancel via $d^r$, for some $r\geq1$, with a term of rank two. However, by \eqref{iso} and \eqref{two} there is at most one copy of $\Q$ in the bigradings $(q,2)$ for $q\geq3$ and $\kh(\overcrossing)$ vanishes in $u=-1$, so no such term exists. Hence it must be the case that the surviving generators live in bigradings $(\pm1,0)$. 

Again by assumption to the contrary, there is a copy of $\Q$ in $(1,1)$ which must now die in the spectral sequence. Since it cannot cancel with the surviving generator in bigrading $(-1,0)$, it must cancel with a generator in $(q,2)$ for $q\geq 3$. Yet \eqref{iso} implies that $\kh(\overcrossing)$ vanishes in $u=2$, so no such generator exists. It must be the case that $\kh^{1,0}(\overcrossing) \isom \Q$ and $ \kh^{1,1}=0$.

Let us now assume that the second outcome of \eqref{two} holds. The two $\Q$ summands in gradings $(3,0)$ and $(3,1)$ must die in the spectral sequence. There are no incoming $d^r$ differentials from gradings $u=0$ or $u=-1$ otherwise a surviving generator is killed. And again by \eqref{iso} there are no terms in the bigradings $(q,2)$ or $(q,1)$ for $q>3$ with which they may cancel. It must be the case that $\kh^{3,0}(\overcrossing)=0$ and  $\kh^{3,1}=0$, and we conclude that $\kh(\overcrossing)$ is thin. The closed formula \eqref{formula} follows immediately from the discussion above.
\end{proof}

\begin{table}
\caption{A portion of the $E^1$ page of the spectral sequence with $u$-grading vertically and $q$-grading horizontally. The induced differentials $d^r$ for $r\geq1$ map the regions in question to the gray regions, whereas the incoming differentials come from the yellow regions.}
\label{E1}
\centering
	$\begin{array}{|c||c|c|c|c|c|c|c|c}
	\hline
	(n-5)+1 & & & & & & \Q & \Q  \\
	\hline
	 (n-5) & & & & & \Q & & \\
	 \hline
	\vdots	& &&&&&& \\
	\hline
	2	&&& \cellcolor{gray!25} & \cellcolor{gray!25} & \cellcolor{gray!25} & \cellcolor{gray!25} & \cellcolor{gray!25} \\
	\hline
	1 && \Q\text{ or } 0 & \cellcolor{gray!25} \Q\text{ or } 0 &\cellcolor{gray!25}&\cellcolor{gray!25} &\cellcolor{gray!25}& \cellcolor{gray!25}  \\
	\hline
	0	& \cellcolor{yellow!25}\Q	 &  \cellcolor{yellow!25} \Q^2\text{ or } \Q & \Q\text{ or } 0 &&&& \\
	\hline
	\hline
	 & -1 & 1 & 3 & \qquad \hdots \qquad & 2(n-5)-1 & 2(n-5)+1 &  \hdots \\
	 \hline
	\end{array} $
\end{table}

\begin{remark}
\label{z2thin}
Important to our application is the fact that the branched double cover of a reduced Khovanov thin knot with $\Ztwo$--coefficients is an L-space, which follows from the symmetry of Heegaard Floer homology under orientation reversal and the spectral sequence from reduced Khovanov homology of a link to the Heegaard Floer homology of the branched double cover of the mirror of the link \cite{OS:BDCs}. Notice that in the argument of Lemma \ref{khthin}, there is a single location not contained on the diagonals $\delta=\pm1$, and this is bigrading $(3,0)$. Had we used $\Z$-coefficients to write down the skein exact sequence, we would have seen
\begin{equation*}
\label{repeat}
0\longrightarrow \kh^{3,0}(\overcrossing) \longrightarrow q\Z^{(2,0)} \stackrel{\cdot u}{\longrightarrow} \cdots.
\end{equation*}
\end{remark}
Since $\kh^{(3,0)}(\overcrossing)$ injects, it is torsion-free, and the argument of Lemma \ref{khthin} shows there are no free summands in bigrading $(3,0)$. Thus $\kh(K_n;\Ztwo)$ is also thin, and therefore $K_n$ is reduced Khovavnov thin with $\Ztwo$ coefficients as well. We deduce that $\Sigma(K_n)$ is an L-space for all $n$.

\subsection{Proof of main theorem}
\label{sec:mainproof}
We now set about to prove that the infinite family of knots $\mathcal{K}$ satisfies the cosmetic crossing conjecture, amongst several other properties. Our main obstruction for a knot to admit a cosmetic crossing change is  

\begin{theorem}\cite[Theorem 2]{LidmanMoore}
Let $K$ be a knot in $S^3$ whose branched double cover $\Sigma(K)$ is an L-space. If each summand of the first singular homology of $\Sigma(K)$ has square-free order, then $K$ admits no cosmetic crossing changes.
\end{theorem}
 
With this obstruction in hand, we prove 

\begin{theorem}
\label{mainexample}
The set $\mathcal{K}$ describes an infinite family of knots which have determinant 49 and non-cyclic $H_1(\Sigma(K);\Z)$. These knots are non-alternating, non-fibered, hyperbolic, of genus two, bridge number three, and satisfy the cosmetic crossing conjecture. 
\end{theorem}

\begin{proof}
By Theorem \ref{hfk}, for all $n$ even, $\hfk(K_n)\isom \hfk(5_2)\otimes \hfk(m(5_2))$ and for all $n$ odd, $\hfk(K_n) \isom \hfk(K_1)$. The knot Floer groups for $\hfk(K_n)$ for $n$ odd can be found after identifying $K_1$ as the alternating knot $10_{22}$, whose knot Floer homology is determined by its Alexander polynomial and signature. Represented as a Poincar\'e polynomial, the knot Floer homology groups are thus
\begin{equation}
\label{even}
P_{\hfk(K_n)}(s,m) = 4s^{-2}m^{-2} +12s^{-1}m^{-1} + 17 +12sm +4s^2m^2
\end{equation}
for $n$ even, and
\begin{equation}
\label{odd}
P_{\hfk(K_n)}(s,m) = 2s^{-3}m^{-3} + 6s^{-2}m^{-2} +10s^{-1}m^{-1} + 13 +10sm + 6s^2m^2 +2s^3m^3
\end{equation}
for $n$ odd. 
By Corollary \ref{fibergenus}, $K_n$ is non-fibered for all $n$ and of genus two when $n$ is even and genus three when $n$ is odd. 

Equations \ref{even} and \ref{odd} imply that $\det(K_n)=49$ for all $n$. This is also implied by \cite[Theorem 2.6]{Lamm:SymmetricRibbon} as well as Lemma \ref{bdc} below. Because there are only a finite number of alternating knots with any fixed determinant (see for instance \cite[Lemma 14]{MooreStarkston}), the knots $K_n$, for $n\in\Z$, are generically non-alternating. Lemma \ref{khthin} and Remark \ref{z2thin} imply $K_n$ is reduced Khovanov homology thin over $\Z/2\Z$ for all $n$, ensuring that $\Sigma(K_n)$ is an L-space for all $n$. 

The rest of the proof will follow after we verify Lemmas \ref{bdc}, \ref{bridge}, and \ref{hyperbolic}.

\begin{lemma}
\label{bdc}
For each $n$, the knot $K_n$ has $H_1(\Sigma(K_n);\Z) \cong \mathbb{Z}/7\mathbb{Z} \oplus \mathbb{Z}/7\mathbb{Z}$ if $n$ is a multiple of $7$ and $H_1(\Sigma(K_n);\Z) \cong \mathbb{Z}/49\mathbb{Z}$ otherwise. 
\end{lemma}
\begin{proof}
Recall that the Goeritz matrix associated to a checkerboard coloring of a knot diagram gives a presentation matrix for $H_1(\Sigma(K_n);\Z)$ \cite{GordonLitherland}. Indeed, to compute the Goeritz matrix of a knot diagram $K$, enumerate the white regions of a checkerboard coloring of $K$ by $X_1,\dots, X_m$, and define the symmetric $m\times m$ integral matrix $G'(K) = (g_{ij})$ by
\begin{equation*}
  	g_{ij} = \left\{
    	\begin{array}{ll}
	-\sum_{c\in X_{ij}}\eta(c) & i\neq j \\
	-\sum_{\ell\neq i} g_{i\ell} & i=j, 
    	\end{array} \right.
\end{equation*}
where the incidence numbers $\eta(c)$ are assigned as in Figure \ref{crossing} and $X_{ij} =  \overline{X}_i\cap\overline{X}_j$. The Goeritz matrix $G:=G(K)$ is then obtained by deleting the first row and column of $G'(K)$. It provides a presentation for $H_1(\Sigma(K);\Z)$ and $\det(K)=|\det{G}|$.
\begin{figure}
	\labellist
	\pinlabel $+1$ at 0 20
	\pinlabel $-1$ at 60 20
	\endlabellist
\includegraphics[width=3cm]{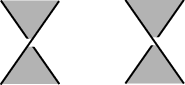}
\caption{Incidence numbers $\eta(c)$ assigned to each crossing in a checkerboard coloring.}
\label{crossing}
\end{figure}
From the diagram in Figure~\ref{sym52}, we obtain a Goeritz matrix presentation for $H_1(\Sigma(K_n);\Z)$,
\begin{equation*}
	\left( \begin{array}{ccccccc}
	4 & 0 & -1 & 0 \\
	0 & -4 & 0 & 1 \\
	-1 & 0 & 2-n & n \\
	0 & 1 & n & -n-2
	\end{array} \right).
\end{equation*}
It is straightforward to verify that this is equivalent to the presentation matrix
$
	\left( \begin{array}{ccc}
	7 & 4n \\
	0 & 7
	\end{array} \right).
$
This presents $\Z/7\Z \oplus \Z/7\Z$ if and only if $7$ divides $4n$, which is equivalent to $n$ being a multiple of $7$.  Otherwise, the matrix presents $\Z/49\Z$. 
\end{proof}
As in \cite{GW:Turaev}, we adopt the strategy of \cite[Lemmas 4 and 5]{Kanenobu:Infinitely} for the following two arguments.
\begin{lemma}
\label{bridge}
All $K_n$ with $n\equiv 0\pmod{7}$ have bridge number three. 
\end{lemma}
\begin{proof}
From the diagram in Figure \ref{sym52}, the $b(K_n)$ is bounded above by three for all $n$. Lemma \ref{bdc} implies that whenever $n \equiv 0 \pmod{7}$, the branched double cover of $K_{n}$ cannot be a lens space because its first homology is non-cyclic, thus $K_n$ cannot be a two-bridge knot by Hodgson-Rubinstein \cite{HR}. Alternatively, recall that two-bridge knots are alternating, so at most finitely many $K_n$ are two-bridge anyway.
\end{proof}

\begin{lemma}
\label{hyperbolic}
All $K_n$ with $n\equiv 0 \pmod{14}$ and $n\neq0$ are hyperbolic.
\end{lemma}
\begin{proof}
By Lemma \ref{bridge}, $b(K_n)=3$ and since $n$ is even, $g(K_n)=2$. By Riley \cite{Riley}, a three-bridge knot is either hyperbolic, a torus knot, or a connected sum. The only torus knot of genus two is the $(5,2)$--torus knot, and its Alexander polynomial distinguishes it from $K_n$ for all $n$. Suppose now $K_n$ is composite. Then $K_n = K \# K'$ for some knots $K$ and $K'$ each of genus one. Since
\begin{equation*}
	br(K\#K') = br(K) + br(K') -1,
\end{equation*}
this implies that $K$ and $K'$ are both two-bridge knots. The branched double cover $\Sigma(K\#K')$ is a nontrivial connected sum of lens spaces with $|H_1(\Sigma(K\#K';\Z)|=49$, so each summand must have order 7. By \cite[Proposition 12.26]{BZ}, a genus one, two-bridge knot or its mirror is of the form $\mathfrak{b}(\alpha, \beta)$, where 
\begin{equation*}
	\beta = 2c,  \qquad \alpha = 4bc \pm 1, \qquad b,c,\in \Z.
\end{equation*}
The branched double cover of $\mathfrak{b}(\alpha, \beta)$ is $L(\alpha, \beta)$, therefore $7=|H_1(L(\alpha,\beta);\Z)|=4bc\pm1$. The only integral solutions are when $b,c = 2,1$ or $b,c=1,2$, both of which correspond with the knot $5_2$ or its mirror. In this case the Jones polynomial distinguishes $K_n$ from connected sums of $5_2$ with itself or its mirror. There can be no such $K$ and $K'$, hence $K_n$ is prime. Since $K_n$ is neither a torus knot nor a connected sum, it must be hyperbolic. 
\end{proof}

Excluding $K_0$ and the finitely many knots $K_n$ which may be alternating, the properties in the statement of Theorem \ref{mainexample} are simultaneously satisfied whenever $n\equiv0\pmod{14}$. This completes the proof of the theorem. 
\end{proof}


\subsection{Observations}
\label{sec:observations}
We close with several observations.
\begin{remark}
For the sake of concreteness, we chose the knot $5_2$ with which to construct the set in \eqref{family}. However, one can carry out similar constructions using other base knots. For example, the partial knot of the pretzel knot $P=(p,q,-p)$, where $p$ is odd (see Figure \ref{fig:bandnumberone}), is the $(2,p)$--torus knot, which is reduced Khovanov homology thin. The general strategy of Lemma \ref{khthin} applies and was carried out by Starkston to investigate their Khovanov homology in \cite{Starkston}. A computation similar to that of Lemma \ref{bdc} would show that $H_1(\Sigma(P);\Z)\isom \Z_p\oplus \Z_p$ if and only if $p\mid q$. In this case, Theorem \ref{ccc} applies when $p$ is square-free, and we similarly obtain that such knots satisfy the cosmetic crossing conjecture. However, when $q$ is odd, $P$ is genus one, and when $q$ is even, $P$ is fibered. So no new information is gained with these pretzel knots, unlike the knots $K_n\in\mathcal{K}$. 
\end{remark}

\begin{remark}
The symmetric unions $K_n$, as well as the symmetric pretzel knots and Kanenobu knots have constant determinant and are Khovanov homology thin. Greene conjectured that there exist only finitely many quasi-alternating links with a given determinant \cite[Conjecture 3.1]{Greene}. We suspect that the present examples, like the Kanenobu knots and pretzel knots, also fail to be quasi-alternating, and that an argument similar to that made by Greene and Watson in \cite{GW:Turaev} for the case of the Kanenobu knots can be made.
\end{remark}

\begin{remark}
Recall that an \emph{L-space knot} is a knot which admits a positive Dehn surgery to an L-space. Because the knots $K_n$ are obtained by rational tangle replacement in $K_0=5_2\#m(5_2)$, there exists a knot $\tilde{\gamma}$ in $\Sigma(K_0)$ which admits surgeries to the L-space $\Sigma(K_n)$ for all $n$. In particular, this knot $\tilde{\gamma}$ is the lift of a crossing arc $\gamma$ in the trivial $0$--tangle $T_0\subset K_0$. Since $\Sigma(K_0)$ is the connected sum of lens spaces $ L(7,2)\# L(7,3)$, we therefore observe that the lift $\widetilde{\gamma}\in  L(7,2)\# L(7,3)$ is an example of an L-space knot in a reducible L-space. An alternate proof to Lemma \ref{bdc} could be obtained by studying presentation matrices for $H_1(\Sigma(K_n);\Z)$ where $\Sigma(K_n)$ is obtained by Dehn surgery along the primitive curve $\widetilde{\gamma}\in  L(7,2)\# L(7,3)$.
\end{remark}

\section*{Acknowledgement}
The author is especially grateful to Tye Lidman for his interest. She also thanks Laura Starkston and Liam Watson for helpful correspondence. This work is supported by NSF grant DMS-1148609.

\bibliographystyle{alpha}
\bibliography{bibliography}

\end{document}